%% file: OppenSchur.tex
\documentclass{amsart}

\usepackage[bookmarks=true,bookmarksnumbered=true,bookmarkstype=toc,pdfusetitle]
{hyperref}
\usepackage{mathtools,amssymb,amsthm}
\mathtoolsset{showonlyrefs,showmanualtags}
\usepackage[abbrev]{amsrefs}
\usepackage{tikz-cd}
\usepackage{bm,bbm}
\usepackage[shortlabels]{enumitem}
\setlist[enumerate]{label=\textup{(\roman*)}}
\AtBeginDocument{
	\def\MR#1{}
}

\theoremstyle{plain}
\newtheorem{thm}{Theorem}[section]
\newtheorem{lem}{Lemma}[section]
\newtheorem{cor}{Corollary}[section]
\newtheorem{prp}{Proposition}[section]
\theoremstyle{definition}
\newtheorem{defn}{Definition}[section]
\theoremstyle{remark}
\newtheorem{rem}{Remark}[section]

\newcommand{\calH}{\mathcal{H}}
\newcommand{\calK}{\mathcal{K}}
\newcommand{\calL}{\mathcal{L}}
\newcommand{\calM}{\mathcal{M}}
\newcommand{\ol}{\overline}
\newcommand{\al}{\alpha}

\newcommand{\gam}{\gamma}
\newcommand{\eps}{\epsilon}
\newcommand{\lam}{\lambda}
\newcommand{\sig}{\sigma}
\newcommand{\imp}{\implies}
\newcommand{\incl}{\hookrightarrow}
\newcommand{\ninf}{{n\to\infty}}
\newcommand{\ex}{\exists}
\newcommand{\es}{\emptyset}
\newcommand{\fa}{\forall}
\newcommand{\sm}{\setminus}
\newcommand{\C}{\mathbb{C}}
\newcommand{\N}{\mathbb{N}}
\newcommand{\R}{\mathbb{R}}
\newcommand{\Llarrow}{\Longleftarrow}
\newcommand{\inner}[1]{\langle#1\rangle}
\newcommand{\biginner}[1]{\bigl\langle#1\bigr\rangle}
\newcommand{\blt}{\ensuremath{\bullet}}
\newcommand{\1}{\mathbbm{1}}
\DeclareMathOperator{\ev}{ev}
\DeclareMathOperator{\ran}{ran}
\DeclareMathOperator{\rank}{rank\,}
\DeclareMathOperator{\re}{Re\,}
\DeclareMathOperator{\Span}{span\,}

\begin{document}

	\title{Oppenheim-Schur's inequality and RKHS}
	\author{Akira Yamada}
	\date{}
	\email{yamada@u-gakugei.ac.jp}
	\pagestyle{headings}

\begin{abstract}
	In 2012, we obtained Oppenheim's inequality for positive semidefinite matrices
	and its equality condition by the reproducing kernel method.
	In this paper, as a continuation, we give a reproducing kernel
	proof of the block matrix version of the Oppenheim-Schur's inequality
	and its equality condition in the positive definite case.
\end{abstract}

\maketitle

\section{Introduction}

Let $A=(a_{ij})$ and $B=(b_{ij})\in M_n$ be $n\times n$ positive semidefinite complex matrices.
The matrix $(a_{ij}b_{ij})\in M_n$ is called the {\em Hadamard product} or the
{\em Schur product} of $A$ and $B$, and is denoted by $A\circ B$.

\begin{thm}[e.g.~\cite{Oppenheim30}]
	The following inequalities hold:
	\begin{enumerate}
		\item Hadamard's inequality:
			\begin{equation} \label{neq:Hadamard}\noeqref{neq:Hadamard}
				|A|\le a_{11}a_{22}\dots a_{nn}.
			\end{equation}
		\item	Oppenheim's inequality:
				\begin{equation} \label{neq:Oppenh}\noeqref{neq:Oppenh}
					|A\circ B|\ge |A|b_{11}\cdots b_{nn}.
				\end{equation}
		\item Oppenheim-Schur's inequality:
				\begin{equation} \label{neq:OppenSchur}\noeqref{neq:OppenSchur}
					|A\circ B|+|A||B|\ge|A|b_{11}\cdots b_{nn}+|B|a_{11}\cdots a_{nn}.
				\end{equation}
	\end{enumerate}
\end{thm}

Our aim is, as a continuation of~\cite{Yamada12}, to prove an extension of the
Oppenheim-Schur inequality for positive semidefinite block matrices by the
theory of the reproducing kernel due to Aronszajn~\cite{Aronszajn50}
and Schwartz~\cite{Schwartz64}, and obtain equality conditions in the
positive definite cases.
Our main tool is the identity~\eqref{eq:lam det} between
the determinant of a positive semidefinite matrix
and the minimum norms of solutions of the inner product in a RKHS
with the matrix as the reproducing kernel.
We also make use of a RKHS which is contained in another Hilbert space,
which is called a {\em sub-Hilbert RKHS}.
Positive semidefinite block matrices are considered as the reproducing kernels
of RKHSs,
which are the most basic and fundamental examples of sub-Hilbert RKHS.\@

\section{Interpolation problem and RKHS}

Throughout the paper, all operators are continuous linear,
and all vector spaces are Hilbert spaces unless otherwise stated.
Using the interpolation problem for inner products in RKHS,
we shall give simple proofs of Oppenheim-Schur's inequality for matrices,
the equality conditions,
and its extensions to block matrices (cf.~\cite{Yamada 12},~\cite{ZhangDing09}).
We will make some preparations for this.

First, we recall the theory of the reproducing kernel
due to L.~Schwartz~\cite{Schwartz64}.
To adapt the symbols and the notion of conjugate spaces
to the case of Hilbert spaces,
we denote by $C^*$ the space of continuous antilinear functionals
in a locally convex space $C$ and call it the {\em antidual} of $C$.
If $\calH$ is a Hilbert space, Riesz's representation theorem allows us to identify
$\calH^*$ with $\calH$.
Let the {\em canonical antiduality} $\inner{C, C^*}$ on $C\times C^*$ be
the {\em sesquilinear form}:
\[
	\ol{y(x)}=\inner{x,y}_{C,C^*},\quad x\in C,\ y\in C^*.
\]
Let $C$ be a Hausdorff quasi-complete locally convex space,
and introduce the weak topology $\sig(C^*, C)$ on $C^*$.
If a Hilbert space $\calH$ is a linear subspace of $C$,
and if the inclusion operator $i\colon\calH\incl C$ is continuous,
$\calH$ is called a {\em $C$-RKHS}.
The {\em reproducing kernel} of the $C$-RKHS $\calH$ is the operator $ii^*\in\calL(C^*,C)$
or {\em Schwartz reproducing kernel}, and we call $C$ the {\em ambient space} of $\calH$.
Since $i^*\colon C^*\to\calH$, we have $\ran ii^*\subset\calH$.
The reproducing kernel $k\in\calL(C^*,C)$ of the $C$-RKHS has
the following reproducing property: for any $x\in\calH$ and $c\in C$,
\begin{equation}\label{eq:reprop}
	\inner{x,kc}_\calH=\inner{x,c}_{C,C^*}.
\end{equation}
We call a $C$-RKHS a {\em sub-Hilbert RKHS}
if the ambient space $C$ is a Hilbert space.
In this case, from the above identification, we have
\[
	\inner{x,kc}_\calH=\inner{x,c}_C.
\]
The reproducing property gives $\|kc\|^2=\inner{kc,c}_{C,C^*}$,
thus for every $c\in C^*$,
\begin{equation}\label{def:k pos}
	\inner{kc,c}_{C,C^*}\ge0.
\end{equation}
If an operator $k\in\calL(C^*,C)$ satisfies the above inequality,
then $k$ is called {\em positive definite} and is denoted by $k\gg0$.
A theorem of Schwartz~\cite{Schwartz64}*{Proposition 10} asserts the converse:
for any $k\in\calL^+(C^*,C)$ (i.e., $k\gg0$), there exists a unique $C$-RKHS $\calH$
with the reproducing kernel $k$, which we denote by $\calH_k$.

\begin{defn}
	Let $\calH$ be a Hilbert space and let $V$ be a vector space.
	For a linear map $A\colon\calH\to V$ with closed kernel,
	there exists a unique Hilbert space structure on the subspace
	$\ran A=A(\calH)$ of $V$ such that the linear map
	$A_0\colon x\in \calH\mapsto Ax\in\ran A$ is a coisometry.
	In other words, we define an inner product on $\ran A$ such that,
	for $f,g\in(\ker A)^\perp$,
	\[
		\inner{Af,Ag}_{\calM(A)}=\inner{f,g}_\calH.
	\]
	The space $\ran A$ equipped with this inner product is a Hilbert space
	which is called the {\em operator range} of the map $A$
	and is denoted by $\calM(A)$.
	The norm $\|\cdot\|_{\calM(A)}$ of $\calM(A)$ is called
	the {\em range norm} (cf.~\cite{Sarason94}).
\end{defn}

By definition, we note that if only one of $f$, $g\in\calH$ belongs to
$(\ker A)^\perp$, we have
\[
	\inner{Af,Ag}_{\calM(A)}=\inner{f,g}_\calH.
\]

\begin{defn}[Image of a RKHS]
	If $T\in\calL(C,D)$ is weakly continuous, then
	the operator range $\calM(T|_\calH)$ of $C$-RKHS $\calH$
	is called an {\em image} or an {\em image RKHS} of $\calH$,
	which is denoted by $T_*\calH$.
\end{defn}

The following theorem is the counterpart of the classical theory
of integral transform (see,~\cite{Saitoh88}*{p.~83})
in Schwartz's theory of RKHS.\@

\begin{thm}[\cite{Schwartz64}*{Proposition 21}] \label{th:imageRKHS}
	Let $C$ and $D$ be quasi-complete Hausdorff locally convex spaces,
	and $\calH$ be a $C$-RKHS with the reproducing kernel $k$.
	Let $T\in\calL(C,D)$ be weakly continuous.
	Then, the operator range $T_*\calH$ of the restriction $T|_\calH$ is a
	$D$-RKHS with the reproducing kernel $TkT^*$.
	The norm of an element $x\in T_*\calH$ satisfies the identity:
	\[
		\|x\|_{T_*\calH}=\inf\{\|y\|_\calH\colon Ty=x,\ y\in\calH\}.
	\]
\end{thm}

In general, the image of the reproducing kernel is dense in the RKHS.\@
Schwartz~\cite{Schwartz64}*{Proposition 7 bis} gave the conditions for coincidence
when the ambient space is locally convex.
Here, we describe the related results for the case where the ambient space is a Hilbert space.
There are many equivalent conditions, but we list some them which are related to RKHSs.
\begin{lem}\label{lem:closedran}
	Let $\calH$ and $\calK$ be Hilbert spaces, and let $A\in\calL(H,K)$.
	The following are equivalent:
	\begin{enumerate}
		\item $\ran A$ is closed.
		\item $\ran A^*$ is closed.
		\item $\ran A=\ran AA^*$.
		\item $\ran AA^*$ is closed.
		\item There exists an operator $X\in\calL(K,H)$ with $AXA=A$.
	\end{enumerate}
\end{lem}

\begin{proof}
	The equivalence of (i) and (ii) is the assertion of the closed range theorem itself.
	From the orthogonal decomposition $H=\ol\ran A^*\oplus\ker A$,
	(ii) implies that $\ran A=A(\ran A^*)=\ran AA^*$,	i.e.,~(iii).
	Conversely, (iii) implies that $H=\ran A^*+\ker A$.
	From the orthogonal decomposition
	we have $\ol\ran A^*\subset\ran A^*$.
	Thus, $\ran A^*$ is closed, i.e.,~(ii).
	We have now proved the equivalence of (i)--(iii).
	Obviously, (i) and (iii) imply (iv).
	
	(iv)$\imp$(iii)
	From $\ran AA^*\subset\ran A=A(\ol\ran A^*)\subset\ol\ran AA^*=\ran AA^*$,
	we obtain $\ran AA^*=\ran A$.

	(v)$\imp$(i)
	Suppose that $Ax_n=y_n\to y$ ($\ninf$).
	Then, $Ax_n=AXy_n\to AXy$, which implies that $y=AXy\in\ran A$.
	Thus, $\ran A$ is closed.

	(i)$\imp$(v)
	The operator $X$ is known as a {\em generalized inverse} of $A$.
	It is well known~(e.g.~\cite{Groetsch77}*{Chapter II}) that if $\ran A$ is closed,
	then the Moore-Penrose inverse $A^+\in\calL(K,H)$ satisfies $AA^+A=A$.
\end{proof}

\begin{lem}\label{lem:cl ran}
	Let $\calH_A$ be a sub-Hilbert $C$-RKHS with the reproducing kernel $A$.
	If $A\in\calL^+(C)$ has a closed range,
	then $\calH_A=\ran A$ as a vector space, and we have, for $\fa x,y\in C$,
	\[
		\inner{Ax,Ay}_{\calH_A}=\inner{Ax,y}_C=\inner{x,Ay}_C.
	\]
\end{lem}

\begin{proof}
	By Theorem~\ref{th:imageRKHS} the RKHS $\calH_A$ is given by
	the operator range of $A^{1/2}$.
	From Lemma~\ref{lem:closedran} we have $\ran A^{1/2}=\ran A$.
	Thus, $\calH_A=\ran A$ as a vector space.
	Since $A$ is the reproducing kernel of $\calH_A$, the last identity holds.
\end{proof}

If $C$ is a finite-dimensional Hilbert space,
then the range of $A\in\calL^+(C)$ is closed,
hence as a corollary of Lemma~\ref{lem:cl ran}
we have (cf.~\cite{Yamada12}*{Proposition 2.1}),

\begin{prp}\label{prp:a_ij RKHS}
	Let $A=(a_{ij})=(\bm a_1\ \bm a_2\ \dots\ \bm a_n)\in M_n$
	be a positive semidefinite matrix.
	Then, the RKHS $\calH_A$ on the set $\{1,2,\dots,n\}$ is given by
	the subspace $\ran A$	of $\C^n$ spanned by the column vectors
	$\bm a_1,\dots,\bm a_n$ of $A$,	equipped with the inner product:
	\[
		\inner{Ax,Ay}_{\calH_A}=\inner{Ax,y}_{\C^n}=\inner{x,Ay}_{\C^n}.
	\]
	In particular, $\dim\calH_A=\rank A$.
	The $i$-th column $\bm a_i$	of the matrix $A$ is
	the reproducing kernel for the point $i$, and we have
	$\inner{\bm{a}_j,\bm{a}_i}=a_{ij}\ (i,j=1,\dots,n)$, i.e.,
	the matrix $A$ is the Gram matrix $G(\bm a_1,\dots,\bm a_n)$ of
	the reproducing kernels $\{\bm a_i\}$ of the RKHS $\calH_A$.
\end{prp}

\begin{lem}\label{lem:sum rkhs}
	Let $A$ and $B$ be positive semidefinite operators in a Hilbert space $C$.
	Then, the following hold:
	\begin{enumerate}
		\item $\calH_{A+B}=\calH_A+\calH_B$ as vector spaces,
				and for any $f\in\calH_A$ and $g\in\calH_B$,
				\begin{equation}\label{eq:PythagIneq}
					\|f+g\|_{\calH_{A+B}}^2\le\|f\|_{\calH_A}^2+\|g\|_{\calH_B}^2.
				\end{equation}
				Equality holds if and only if $\inner{f,h}_{\calH_A}=\inner{g,h}_{\calH_B}$
				for every $h\in\calH_A\cap\calH_B$.
		\item If $C$ is finite dimensional, then $\ran\,(A+B)=\ran A+\ran B$.
				Furthermore, the equality in the inequality~\eqref{eq:PythagIneq} holds
				if and only if there exists a $z\in C$ such that $f=Az$ and $g=Bz$.
	\end{enumerate}
\end{lem}

\begin{proof}
	(i)
	Well-known facts (cf.~\cite{Aronszajn50}*{pp.~352--354}).
	One can also prove this fact by showing that
  the operator range of the linear map $f\oplus g\in\calH_A\oplus\calH_B\mapsto f+g\in\C^E$ is the RKHS with the reproducing kernel $A+B$.

	(ii)
	The first assertion follows from Lemma~\ref{lem:cl ran} and (i).
	Let us show the last equivalence.

	$\Llarrow$
	By Lemma~\ref{lem:cl ran} we have, for every $h\in\ran A\cap\ran B$,
	\[
		\inner{f,h}_{\calH_A}=\inner{Az,h}_{\calH_A}=\inner{z,h}_C
		=\inner{Bz,h}_{\calH_B}=\inner{g,h}_{\calH_B}.
	\]
	Thus, equality holds by (i).

	$\imp$
	Since, by assumption, the ranges of the operators $A$ and $B$ are closed,
	Lemma~\ref{lem:cl ran} implies that $\calH_A=\ran A$ and $\calH_B=\ran B$.
	Thus, there exist $x$ and $y\in C$ with $f=Ax$ and $g=By$.
	If equality occurs, then, by (i), we have for each $h\in\ran A\cap\ran B$,
	\[
		\inner{x,h}_C=\inner{f,h}_{\calH_A}=\inner{g,h}_{\calH_B}
		=\inner{y,h}_C.
	\]
	Thus, $x-y\in(\ran A\cap\ran B)^\perp$.
	Since for arbitrary subspaces $E$ and $F$ of $C$, $(E\cap F)^\perp=E^\perp+F^\perp$,
	we have
	\[
		(\ran A\cap\ran B)^\perp=(\ran A)^\perp+(\ran B)^\perp=\ker A+\ker B.
	\]
	Thus, there exist $a\in\ker A$ and $b\in\ker B$ with $x-y=a+b$,
	which implies that $f=Az$ and $g=Bz$ if we take $z=x-a=y+b\in C$.
\end{proof}

Using the theory of integral transform, we can obtain results for solutions
of the interpolation problems concerning the inner product
of a Hilbert space.

\begin{thm}\label{thm:inner_intp}
	Let $G=(\inner{a_j,a_i})\in M_n$ be the Gram matrix of the sequence $\{a_i\}_{i=1}^n$
	in a Hilbert space $\calH$.
	Given $b=(b_j)\in\C^n$, consider the interpolation problem for $f\in\calH$:
	\begin{equation}\label{eq:interp}
		\inner{f,a_i}=b_i,\quad i=1,\dots,n
	\end{equation}
	Then, the following hold:
	\begin{enumerate}
		\item The RKHS $\calH_G$ consists of the set of vectors $b\in\C^n$ such that
				there exists a solution $f\in\calH$ of~\eqref{eq:interp}.
				$\calH_G=\ran G$ as vector spaces.
		\item The norm of $\calH_G$ is the range norm of the operator
				$f\in\calH\mapsto(\inner{f,a_i})\in\C^n$:
				for any $b\in\calH_G$,
				\[
					\|b\|_{\calH_G}=\inf\{\|f\|\colon\inner{f,a_i}=b_i,\ i=1,\dots,n\}.
				\]
		\item If there exists a solution $f\in\calH$ of~\eqref{eq:interp},
				then there exists a unique solution with the minimum norm.
				The solution $f$ has the minimum norm if and only if $f\in\Span\{a_i\}_{i=1}^n$.
	\end{enumerate}
\end{thm}

\begin{proof}
	(i)
	Let $E$ be the set $\{1,2,\dots,n\}$.
	For $f\in\calH$, define a function $\hat f\colon E\to\C$
	by $\hat f(i)=\inner{f,a_i}$, $i\in E$.
	Then, the space $\hat\calH=\{\hat f\colon f\in\calH\}$ is an RKHS on $E$ with
	the reproducing kernel $G$, which is called	the {\em integral transform}
	of $\calH$ (see,~\cite{Saitoh88}*{p.~83}).
	On the other hand, by Theorem~\ref{lem:cl ran},
	if $G$ is a positive semidefinite matrix,
	then $\ran G$ is an RKHS on $E$ with the reproducing kernel $G$.
	By the uniqueness of the RKHS with the same reproducing kernel,
	we have $\hat\calH=\ran G$.
	Therefore, the set of $(b_i)_{i=1}^n\subset\C^n$ with a solution of the
	interpolation problem~\eqref{eq:interp} coincides with the set $\ran G$
	of the image of the integral transform.

	(ii)
	Since the image RKHS $\ran G$ of the integral transform is the operator range
	of the map $f\in\calH\mapsto\hat f\in\hat\calH$, the norm of $(b_i)\in\ran G$
	coincides with the range norm of $\hat\calH$.

	(iii)
	For a solution $f$ of~\eqref{eq:interp},
	let $f_n$ be the orthogonal projection of $f$ onto $\Span\{a_i\}$.
	Then, $f_n$ is also a solution, and since $f-f_n\in\{a_1,\dots,a_n\}^\perp$, we have
	\[
		\|f\|^2=\|f-f_n\|^2+\|f_n\|^2.
	\]
	Thus, $f_n$ is the unique solution with the minimum norm.
\end{proof}

If the sequence $\{a_i\}$ is linearly independent, then we obtain
a concrete representation of the minimum norm solution using the determinant.

\begin{thm}[cf.~\cite{AkhiezerGlazman93}*{p.~13},~\cite{Yamada12}]\label{thm:minimum_int}
	Let $\{a_i\}_{i=1}^n$ be a linearly independent subset of a Hilbert space $\calH$.
	Then, for any $(b_i)_{i=1}^n\in\C^n$, there exists a unique element
	$f\in\calH$ which satisfies	the interpolation
	problem~\eqref{eq:interp} and minimizes	the norm,
	where $f$ and its norm are given by
	\begin{align} \label{eq:f expr}
		f
		&=-\frac1{|G_n|}
		\begin{vmatrix}
			0      & a_1             & \cdots  & a_n \\
			b_1    & \inner{a_1,a_1} & \cdots  & \inner{a_n,a_1} \\
			\vdots & \vdots          & \ddots  & \vdots \\
			b_n    & \inner{a_1,a_n} & \cdots  & \inner{a_n,a_n}
		\end{vmatrix}, \\
		\|f\|^2
		&=-\frac1{|G_n|}
		\begin{vmatrix} \label{eq:norm_f_n}
			0      & \ol{b}_1        & \cdots  & \ol{b}_n \\
			b_1    & \inner{a_1,a_1} & \cdots  & \inner{a_n,a_1} \\
			\vdots & \vdots          & \ddots  & \vdots \\
			b_n    & \inner{a_1,a_n} & \cdots  & \inner{a_n,a_n}
		\end{vmatrix}.
	\end{align}
\end{thm}

To derive determinant inequalities,
we now consider a special interpolation problem~\eqref{eq:interp}
of the form
\begin{equation}\label{eq:b_n}
	b_1=\dots=b_{n-1}=0,\ b_n=1,
\end{equation}
assuming that the set $\{a_i\}\subset\calH$ is linearly independent.
Let $f_n$ be the minimum norm solution for the above data~\eqref{eq:b_n},
and $\lam_n=\|f_n\|$ be its norm.
From Theorem~\ref{thm:minimum_int} and~\eqref{eq:b_n}
we conclude that the sequence $\{f_n\}$ of solutions with minimum norm coincides
up to multiple constants with the sequence $\{F_n\}$ obtained by the Gram-Schmidt
orthogonalization process.
More precisely, $f_n=(|G_{n-1}|/|G_n|)^{1/2}F_n$, $n=1,2,\dots$.
Hence, we have

\begin{cor}\label{cor:minnorm}
	If $\{a_i\}_{i=1}^n\subset\calH$ is linearly independent,
	then $\lam_n=\sqrt{|G_{n-1}|/|G_n|}$,
	where $G_k=G(a_1,\dots,a_k)$, $k=1,\dots,n$, $G_0=1$ are the Gram matrices of $\{a_i\}$.
\end{cor}

\section{Hadamard product of vector-valued RKHSs and inner product interpolation}

Consider a $C$-RKHS $\calH$ whose ambient space $C$ is a direct sum
$\bigoplus_{x\in E}C_x$ of Hilbert spaces $\{C_x\}_{x\in E}$.
An element of $C=\bigoplus_{x\in E}C_x$ is regarded as a
function that takes values in $C_x$ at the point $x\in E$.
The point evaluation $f\in\calH\mapsto f(x)\in C_x$ at $x\in E$ is denoted by
$\ev_x$, and, for each $x\in E$, $\ev_x\in\calL(\calH,C_x)$ is bounded.
The adjoint operator $k_x=\ev_x^*\in\calL(C_x,\calH)$ is called the 
{\em reproducing kernel of $\calH$ for the point $x$}.
The reproducing property of $k_x$ is given by, for $f\in\calH$, $c\in C_x$,
\[
	\inner{f,k_x(c)}_\calH=\inner{f(x),c}_{C_x}.
\]
The operator $K(x,y)=k_x^*k_y\in\calL(C_y,C_x)$ is called the {\em kernel function}
of $C$-RKHS $\calH$.
Between the Schwartz reproducing kernel $K\in\calL^+(C)$ of $C$-RKHS $\calH$,
the reproducing kernel $k_x$ for the point $x$, and the kernel function $K(x,y)$,
the following identities hold:
\[
	\ev_xKi_y=K(x,y),\quad k_y=Ki_y,
\]
where $i_y\colon C_y\incl C$ denotes the canonical injection.
The kernel function $K(x,y)$ of $C$-RKHS $\calH$ is {\em positive semidefinite}:
for every $a_1,\dots,a_n\in E$ and $c_i\in C_i$,
\[
	\sum\inner{K(a_i,a_j)c_j,c_i}_{C_{a_i}}\ge0.
\]
Note that the inner product of the {\em tensor product Hilbert space}
$\bigotimes_{p=1}^mH_i$ of the Hilbert spaces $H_1,\dots,H_m$ satisfies the identity
(see, e.g.~\cite{ReedSimon72}*{p.~49}): for $f_i\in H_i$, $g_i\in H_i$, $p=1,\dots,m$,
\begin{equation} \label{eq:tens inner}
	\inner{f_1\otimes\dots\otimes f_m,\,g_1\otimes\dots\otimes g_m}_{\otimes H_i}
	=\prod_{p=1}^m\inner{f_i,g_i}_{H_i}.
\end{equation}

Hereafter, we fix the natural numbers $m$ and $s$,
and let $E$ be the set $\{1,\dots,s\}$ and $C^p=\bigoplus_{j=1}^sC_j^p$
be the direct sum of the finite-dimensional Hilbert space $C_j^p$.
For simplicity, we identify an element of $C_j^p$ with its image in $C^p$ of
the canonical injection.
Under this identification, $C_j^p$ is a mutually orthogonal subspace of $C^p$,
which leads to the orthogonal decomposition $C^p=\bigoplus_{j=1}^sC_j^p$.
Now we have a canonical isomorphism of the tensor product Hilbert space
$\bigotimes_{p=1}^mC^p$:
\[
	\bigotimes_{p=1}^mC^p
	\cong\bigoplus_{(j_1,\dots,j_m)\in E^m}C_{j_1}^1\otimes\dots\otimes C_{j_m}^m.
\]\
Let $\calH^p\ (p=1,\dots,m)$ be a $C^p$-RKHS on $E$,
and let $k_j^p\in\calL(C^p,\calH^p)$ be the reproducing kernel
of $\calH^p$ for the point $j\in E$.
Since every element of $\calH^p$ is a $C^p$-valued function on $E$,
for $f_p\in\calH^p$, $p=1,\dots,m$,
we can define the simple tensor $f_1\otimes\dots\otimes f_m$ as a vector-valued
function on the Cartesian product $E^m$:
\[
(f_1\otimes\dots\otimes f_m)(j_1,\dots,j_m)
=f_1(j_1)\otimes\dots\otimes f_m(j_m)\in\bigotimes_{p=1}^mC^p.
\]
Thus, the tensor product Hilbert space $\bigotimes_{p=1}^m\calH^p$
is a $\bigotimes_{p=1}^mC^p$-RKHS on the product space $E^m$.
From the identity~\eqref{eq:tens inner}, for any $\otimes_{p=1}^mf_p\in\bigotimes\calH^p$
and $c_{j_p}^p\in C^p$, we have
\begin{align*}
	\inner{\otimes_{p=1}^mf_p,\otimes_{p=1}^mk_{j_p}^p(c_{j_p}^p)}_{\otimes\calH^p}
	&=\prod_{p=1}^m\inner{f_p(j_p), c_{j_p}^p}_{C_j^p}
		=\inner{\otimes_{p=1}^mf_p(j_p),\otimes_{p=1}^mc_{j_p}^p}_{\otimes_{p=1}^mC_{j_p}^p} \\
	&=\inner{(\otimes_{p=1}^mf_p)(j_1,\dots,j_m),
		\otimes_{p=1}^mc_{j_p}^p}_{\otimes_{p=1}^mC_{j_p}^p},
\end{align*}
which shows that $\otimes_{p=1}^mk_{j_p}^p$ is the reproducing kernel of
$\bigotimes_{p=1}^m\calH^p$ for the point $(j_1,\dots,j_m)\in E^m$.
Let
\[
	\phi^*\colon f\in\bigotimes_{p=1}^m\calH^p\mapsto
		\phi^*f=f\circ\phi\in\bigodot_{p=1}^mC^p
\]
be the pullback by the diagonal map $\phi\colon j\in E\mapsto (j,\dots,j)\in E^m$,
where $\bigodot_{p=1}^mC^p$ denotes the direct sum
\[
	\bigoplus_{j\in E}C_j^p\otimes\dots\otimes C_j^m.
\]
Denote the operator range of $\phi^*$ by $\bigodot_{p=1}^m\calH^p$,
which is a $\bigodot_{p=1}^mC^p$-RKHS on $E$,
and is called the {\em Hadamard product RKHS} of $\{\calH^p\}_{p=1}^m$.
For vector-valued functions $f_1\in\calH^1,\dots,f_m\in\calH^m$,
the pullback $\phi^*(f_1\otimes\dots\otimes f_m)$ is denoted by
$\bigodot_{p=1}^mf_i$ or $f_1*\dots*f_m$,
and satisfies $\bigodot_{p=1}^mf_i\in\bigodot_{p=1}^m\calH^p$.
This is called the {\em Hadamard product} of $\{f_p\}_{p=1}^m$.
Since $\bigotimes_{p=1}^mk_j^p\in(\ker\phi^*)^\perp$,
by the reproducing property of the tensor product $\bigotimes_{p=1}^mk_j^p$,
and by definition of the operator range, we see that
$\bigodot_{p=1}^mk_i^p\in\bigodot_{p=1}^m\calH^p$ is the reproducing kernel for the point $i\in E$ of $\bigodot_{p=1}^m\calH^p$.
Furthermore, we immediately have an inequality,
for any $f\in\bigotimes_{p=1}^m\calH^p$,
\begin{equation}\label{eq:rest_ineq}
	\|\phi^*f\|_{\bigodot_{p=1}^m\calH^p}\le\|f\|_{\bigotimes_{p=1}^m\calH^p}.
\end{equation}

\begin{defn}
	An element $f\in\bigotimes_{p=1}^m\calH^p$ is called {\em extremal}
	if $m\ge2$ and if equality holds in the above inequality (\cite{Yamada09}*{p.~378}).
\end{defn}

Note that $f$ is extremal if and only if $f\in(\ker\phi^*)^\perp$.

\begin{prp}[cf.~\cite{Yamada12}] \label{prp:extremal}
	Let $C^p$ be a Hilbert space, and let $\calH^p$ be a $C^p$-RKHS on the set $E$
	($p=1,\dots,m$).
	An element $f\in\bigotimes_{p=1}^m\calH^p$ is extremal if and only if
	$f$ is in the closed linear span of the set
	$\bigcup_{j\in E}\ran k_j^1\otimes\dots\otimes\ran k_j^m$,
	where $k_j^p$ is the reproducing kernel of $\calH^p$ for the point $j\in E$.
\end{prp}

\begin{proof}
	By the reproducing property of $\bigotimes_{p=1}^mk_j^p$ for the
	point $(j,j,\dots,j)\in E^m$, we have
	$\ker\phi^*=(\bigcup_{j\in E}\bigotimes_{p=1}^m\ran k_j^p)^\perp$.
	Thus, its orthogonal complement is given by
	$(\ker\phi^*)^\perp=\bigvee_{j\in E}\bigotimes_{p=1}^m\ran k_j^p$,
	where $\bigvee$ denotes a closed linear span.
\end{proof}

In particular, if the ambient space is finite dimensional,
then we have a simple characterization of extremal simple tensors.

\begin{prp}\label{prp:tensCiDi}
	Let $m\ge2$.
	Suppose that $C^p=\bigoplus_{i=1}^sC_i^p$ is a direct sum of finite dimensional
	Hilbert spaces, and suppose that each $A^p\in\calL(C^p)$ is positive definite
	($A^p>0$) for every $p=1,\dots,m$.
	If $f_p\in\calH^p\sm\{0\}$ for $p=1,\dots,m$, then the simple tensor
	$f_1\otimes\dots\otimes f_m$ is extremal if and only if
	there exists a point $i\in E$ such that
	$f_p\in\ran k_i^p$ for every $p=1,\dots,m$.
\end{prp}

\begin{proof}
	For each $p$, the reproducing kernel $k_i^p$ for the point $i\in E$
	is the $i$-th column of the operator matrix $A^p=(A_{ij}^p)_{i,j=1}^s$,
	that is, $A^p=(k_1^p\ \dots\ k_s^p)$.
	Since $A^p$ is positive definite, $A^p$ is an injection.
	Thus, if $A^pc=\sum_{i=1}^sk_i^pc_i=0$ for $c=(c_i)\in C^p$,
	then $c_i=0$ for each $i$.
	Since $C^p$ is finite dimensional, by Lemma~\ref{lem:cl ran}
	$\calH^p$ and $\ran A^p$ are identical as vector spaces.
	Hence, we conclude that $\calH^p$ is a direct sum $\bigoplus_{i=1}^s\ran k_i^p$.

	$\imp$
	Since for every $p$ $\calH^p$ is a direct sum $\bigoplus_{i=1}^s\ran k_i^p$
	and since $f_p\ne0$,
	there exists a subset $E_p$ of $E$ and an element $c_i^p\in C_i^p$
	such that $f_p$ is a linearly independent sum of the form
	\[
		f_p=\sum_{i\in E_p}k_i^p(c_i^p).
	\]
	Thus,	their tensor product is the sum of linearly independent terms:
	\[
		f_1\otimes\dots\otimes f_m
		=\sum_{i_1\in E_1,\dots,i_m\in E_m}
			k_{i_1}^1(c_{i_1}^1)\otimes\dots\otimes k_{i_m}^m(c_{i_m}^m).
	\]
	Since $C^p$ is finite dimensional, from Proposition~\ref{prp:extremal},
	$f_1\otimes\dots\otimes f_m$ must be a linear combination of elements
	in $\ran k_i^1\otimes\dots\otimes\ran k_i^m$, $i\in E$.
	This is only possible if there exists an $i\in E$ such that the set $E_p$
	is the singleton $\{i\}$ for every $p=1,\dots,m$.

	$\Llarrow$
	This is clear from Proposition~\ref{prp:extremal}.
\end{proof}

In this paper, the inner product interpolation in a sub-Hilbert RKHS played
an important role.
We will describe the setting and state some lemmas concerning the interpolation problem.
Let $C$ be a Hilbert space.
Given a $C$-RKHS $\calH_A$ and a CONS $\{u_j\}_{j\in J}$ of $C$,
we consider the problem of finding solutions $f\in\calH_A$ that satisfy, for $j\in J$,
\begin{equation}\label{eq:inner ip}
	\inner{f,Au_i}_{\calH_A}=0\ \text{for all }i<j,\ \text{and }\inner{f,Au_j}_{\calH_A}=1,
\end{equation}
which is called the {\em inner product interpolation problem (IPIP)} of order $j$
with respect to a CONS $\{u_j\}$ of the $C$-RKHS $\calH_A$.
For simplicity, we denote this IPIP of order $j$ by triple $(\calH_A,\{u_j\},j)$.
The set of solutions to this problem is denoted by $P_j^A$.
We require that the index set $J$ to be linearly ordered.
If $P_j^A\ne\es$, we denote the element of $P_j^A$ with the minimum norm by $f_j^A$
and its norm by $\lam_j^A$.
If $C=\bigoplus_{i=1}^sC_i$ is a direct sum of Hilbert spaces,
we denote the CONS of $C_i$ with a double suffix as
$\{u_{ij}\colon\ j=1,\dots,n_i\}$.
Thus, $\{u_{ij}\colon\ i=1,\dots,s;\ j=1,\dots,n_i\}$ is a CONS of $C$.
Here, we introduce the lexicographic order for the suffix $ij$ of $u_{ij}$.
From the definition of $f_{ij}^A$ and the reproducing property of
$k_{ij}^A$, we obtain
\[
	f_{ij}^A(l)=0,\quad l=1,\dots,i-1,\ j=1,\dots,n_i.
\]
This is important when considering equality conditions for inequalities.

\begin{lem} \label{lem:main_ineq}
	Let $C=\bigoplus_{i=1}^sC_i$ be the direct sum of finite-dimensional
	Hilbert spaces $C_i$,
	and let the operator $A\in\calL(C)$ be expressed as an operator matrix
	$A=(A_{ij})_{i,j=1}^s$, $A_{ij}\in\calL(C_j,C_i)$.
	If $A$ is positive definite, then the IPIP $(\calH_A,\{u_{ij}\},ij)$
	has a solution, and the following inequality holds:
	\begin{equation}\label{eq:norm lbd}
		\lam_{ij}^A\ge1/\inner{A_{ii}u_{ij},u_{ij}}_{C_i}^{1/2}.
	\end{equation}
	Equality holds if and only if $f_{ij}^A$ is a constant multiple of $k_i^A(u_{ij})$,
	which in turn is equivalent to the condition that
	$k_{i'}^A(u_{i'j'})\perp k_i^A(u_{ij})$ in $\calH_A$ for every $i'j'<ij$.
	In particular, when equality occurs, the $C_1,\dots,C_{i-1}$-components of
	the vector $Au_{ij}\in C$ are equal to 0.
\end{lem}

\begin{proof}
	Since $C$ is finite dimensional, $A>0$ implies that $A$ is an injection.
	Hence, the sequence $\{Au_{ij}\}$ is linearly independent,
	which implies that the Gram matrix $G$ is nonsingular.
	Therefore, $G$ is a surjection, which implies that $P_{ij}^A\ne\es$
	by Theorem~\ref{thm:inner_intp}.
	Since $Au_{ij}=k_i^A(u_{ij})$, if $f\in P_{ij}^A$, then by Schwarz's inequality,
	\[
		1=\inner{f,k_i^A(u_{ij})}_{\calH_A}\le\|f\|_{\calH_A}\|k_i^A(u_{ij})\|_{\calH_A}.
	\]
	From $u_{ij}\ne0$, we conclude that
	\[
		\|k_i^A(u_{ij})\|^2_{\calH_A}=\inner{(k_i^A)^*k_i^A(u_{ij}),u_{ij}}_{C_i}
		=\inner{A_{ii}u_{ij},u_{ij}}_{C_i}>0.
	\]
	Thus,
	\[
		\|f\|_{\calH_A}\ge1/\inner{A_{ii}u_{ij},u_{ij}}_{C_{i}}^{1/2}.
	\]
	Taking the minimum of the left-hand side, we obtain the inequality~\eqref{eq:norm lbd}.
	From the equality condition of Schwarz's inequality, $f_{ij}^A$ and $k_i^A(u_{ij})$
	are linearly dependent, thus we have
	\[
		f_{ij}^A=\frac{k_i^A(u_{ij})}{\inner{A_{ii}u_{ij},u_{ij}}_{C_{i}}}.
	\]
	By definition of $f_{ij}^A$, if $i'j'<ij$, then it is clear that
	$k_{i'}^A(u_{i'j'})\perp k_i^A(u_{ij})$.
	Conversely, if this is the case, then $f_{ij}^A$ is a constant multiple of
	$k_{ij}^A(u_{ij})$, which implies equality.
	The last statement follows from the note immediately above this Lemma.
\end{proof}

If $\{u_j\}_{j=1}^n$ is a CONS of the $n$-dimensional Hilbert space $C$,
then, by Parseval's theorem, the operator $\phi\colon C\to\C^n$
defined by $\phi(f)=(\inner{f,u_j})_{j=1}^n$ is an onto isometry.
We call $\phi$ the isometry {\em induced by} the CONS $\{u_j\}_{j=1}^n$.
From a finite-dimensional sub-Hilbert RKHS $\calH$,
we can construct a classical RKHS due to Aronszajn~\cite{Aronszajn50}
which is isometrically isomorphic to $\calH$.

\begin{lem}\label{lem:scalarztn}
	Let $\phi\colon C\to\C^n$ be the isometry induced by
	a CONS $\{u_j\}_{j=1}^n$, and let $A\in M_n$ be the matrix
	representing the operator $T\in\calL^+(C)$ with respect to the basis $\{u_j\}_{j=1}^n$.
	Then, the matrix $A=\phi T\phi^{-1}$ is positive semidefinite,
	$\calH_T=\ran T$ and $\calH_A=\ran A$ as vector spaces, and we have
	for any $x,y\in\calH_T$,
	\[
		\hspace{4cm}
		\inner{x,y}_{\calH_T}=\inner{\phi x,\phi y}_{\calH_A}.
		\hspace{2cm}
		\begin{tikzcd}
			C\ar[r,"\phi"] & \C^n \\
			\calH_T\ar[r,dashed,"\phi|_{\calH_T}"]\ar[u,hook] & \calH_A\ar[u,hook]
		\end{tikzcd}.
	\]
	Thus, $\phi|_{\calH_T}\colon\calH_T\to\calH_A$ is an onto isometry.
\end{lem}

\begin{proof}
	Since $\phi$ is an isometric isomorphism, we have $\phi^*=\phi^{-1}$.
	By Theorem~\ref{th:imageRKHS}, this implies that the reproducing kernel
	of the image RKHS $\phi_*(\calH_T)$	is the matrix $A$.
	Since $\phi_*(\calH_T)$ is a $C^n$-RKHS, we see that $\phi_*(\calH_T)=\calH_A$.
	By Lemma~\ref{lem:cl ran} $\calH_T=\ran T$ and $\calH_A=\ran A$ as vector spaces,
	since $C$ and $\C^n$ are both finite dimensional.
	Furthermore, for $x=Tz,\ y=Tw$, $z,w\in C$,
	\begin{align*}
		\inner{x,y}_{\calH_T}
		&=\inner{x,w}_C=\inner{\phi x,\phi w}_{\C^n}=\inner{\phi x,A\phi w}_{\calH_A}
			=\inner{\phi x,\phi y}_{\calH_A}.
	\end{align*}
\end{proof}

\begin{defn}
	The RKHS $\calH_A$ on $\{1,\dots,n\}$ constructed above is called the
	{\em scalarization} of the $C$-RKHS $\calH_T$.
\end{defn}

From Lemma~\ref{lem:scalarztn}, we can express the determinant of the reproducing kernel
$T$ of the sub-Hilbert RKHS using minimum norms $\{\lam_i^T\}$.

\begin{lem}\label{lem:prod det}
	Let $\{u_j\}_{j=1}^n$ be a CONS of finite-dimensional Hilbert space $C$,
	and let $T\in\calL(C)$ be	positive definite.
	Then,
	\begin{equation}\label{eq:lam det}
		\prod_{i=1}^n\lam_i^T=|T|^{-1/2}.
	\end{equation}
	Furthermore, if $C$ is the direct sum $C=\bigoplus_{i=1}^sC_i$, and if
	$\{c_{ij}\}_{j=1}^{n_i}$ is a CONS of $C_i$, then, for $i=1,\dots,s$,
	\[
		\prod_{j=1}^{n_i}\lam_{ij}^T=\frac{|T_{i-1}|^{1/2}}{|T_i|^{1/2}},
		\quad(|T_0|=1),
	\]
	where $T_i=(T_{jk})_{j,k=1}^i$ is the $i$-th leading principal submatrix
	of the operator matrix $T=(T_{jk})_{j,k=1}^s$.
\end{lem}

\begin{proof}
	By Lemma~\ref{lem:scalarztn}, let $\phi\colon C\to\C^n$ be the isometry
	induced by a CONS $\{u_j\}$, and let $\calH_A$ be the scalarization of $\calH_T$.
	Then, the representation matrix $A\in M_n$ of $T$ with respect to the basis
	$\{u_j\}$ is positive definite, and $\calH_T$ is isometrically isomorphic with
	$\calH_A$.
	If $\{e_j\}_{j=1}^n$ is the canonical basis of $\C^n$,
	then $\phi(u_j)=e_j$, for each $j=1,\dots,n$.
	Thus, the inner product interpolation problem of $\calH_T$ with respect to
	$\{u_j\}\subset C$ is reduced to $\calH_A$ with respect to $\{e_j\}\subset\C^n$.
	Since $Ae_j$ is the reproducing kernel of $\calH_A$ for the point
	$j\in\{1,\dots,n\}$ by Proposition~\ref{prp:a_ij RKHS},
	we have, from Corollary~\ref{cor:minnorm},
	\[
		|A|=G_n(A)=\frac{G_n(A)}{G_{n-1}(A)}\cdot\frac{G_{n-1}(A)}{G_{n-2}(A)}\cdots G_1(A)
		=(\lam_n^A\cdots\lam_1^A)^{-2}.
	\]
	Thus,~\eqref{eq:lam det} holds by isometry, $\phi|_{\calH_T}$.

	If $C$ is a direct sum, let $\iota_i\colon C^{(i)}\to C$ be the canonical injection,
	and let $\pi_i\colon C\to C^{(i)}$ be the projection with $C^{(i)}=\bigoplus_{j=1}^iC_j$.
	Then, it is easy to see that $T_i\in\calL(C^{(i)})$ is the reproducing kernel
	of the image RKHS $(\pi_i)_*(\calH_T)$,
	and, by the reproducing property, for each $f\in\calH_T$ and $c_{lj}\in C_l$ ($l\le i$),
	\begin{align*}
		\inner{f,k_l^T(c_{lj})}_{\calH_T}
		&=\inner{f,T(\iota_ic_{lj})}_{\calH_T}=\inner{f,\iota_ic_{lj}}_C \\
		&=\inner{\pi_if,c_{lj}}_{C^{(i)}}
			=\inner{\pi_if,T_ic_{lj}}_{\calH_{T_i}} \\
		&=\inner{\pi_if,k_l^{T_i}(c_{lj})}_{\calH_{T_i}}.
	\end{align*}
	Thus, $\lam_{lj}^T=\lam_{lj}^{T_i}$, $l\le i$.
	Consequently, for the CONS $\{u_{ij}\}$	($=\bigcup_{i=1}^s\{u_{ij}\}_{j=1}^{n_i})$ of $C$,
	we have,
	\[
		\prod_{j=1}^{n_i}\lam_{ij}^T
		=\frac{\prod_{l=1}^i\prod_{j=1}^{n_l}\lam_{lj}^T}
			{\prod_{l=1}^{i-1}\prod_{j=1}^{n_l}\lam_{lj}^T}
		=\frac{\prod_{l=1}^i\prod_{j=1}^{n_l}\lam_{lj}^{T_i}}
		{\prod_{l=1}^{i-1}\prod_{j=1}^{n_l}\lam_{lj}^{T_{i-1}}},
	\]
	which is equal to $|T_{i-1}|^{1/2}/|T_i|^{1/2}$ by the first half of the proof.
\end{proof}

\section{Main results}\label{sec:main results}

We list the settings:
\begin{itemize}
	\item $C=\bigoplus_{p=1}^mC^p$, $C^p=\bigoplus_{i=1}^sC_i^p$: direct sums of
			finite-dimensional Hilbert spaces,
	\item $C_i^p$ is identified with the subspace of $C^p$ by the canonical injection,
	\item $\{c_{ij}^p\}_{j=1}^{n_i^p}$: a CONS for $C_i^p$, $i=1,\dots,s$, $p=1,\dots,m$,
	\item The order of the set $J^p$ of subscripts of $\{c_{ij}^p\}$
		is the lexicographic order of $\N^2$ obtained by identifying $ij$ with $(i,j)\in\N^2$:
		\[
		J^p=\{ij\colon 1\le i\le s;\ 1\le j\le n_i^p\},
		\]
		and the order of the Cartesian product $J=\prod_{p=1}^mJ^p$
		is the lexicographic order of $J^p$, $p=1,\dots,m$.
	\item The subset $J_i$ of $J$ is defined by
		\[
			J_i=\{(ij_1,\dots,ij_m)\colon1\le j_p\le n_i^p,\ 1\le p\le m\}.
		\]
\end{itemize}

Now, the set $\{c_{ij}^p\}_{ij\in J^p}$ is a CONS of $C^p$,
and if we put, for $\gam=(ij_1,\dots,ij_m)\in J_i$, $i=1,\dots,s$,
\[
	c_\gam=c_{ij_1}^1\otimes\dots\otimes c_{ij_m}^m\in\bigotimes_{p=1}^mC_i^p,
\]
then $c=\{c_\gam\}_{\gam\in\bigcup_{i=1}^sJ_i}$ is a CONS of the Hadamard product
\[
	\bigodot_{p=1}^mC^p=\bigoplus_{i=1}^s\bigotimes_{p=1}^mC_i^p.
\]
If $A^p=(A_{ij}^p)_{i,j=1}^s\in\calL(C^p)$ is positive semidefinite
for each $p=1,\dots,m$,
then the Hadamard product RKHS $\bigodot_{p=1}^m\calH_{A^p}$ of the
$C^p$-RKHS $\calH_{A^p}$ is a $\bigodot_{p=1}^mC^p$-RKHS on the set $E=\{1,\dots,s\}$,
whose reproducing kernel is given by
\[
	\bigodot_{p=1}^mA^p
	=(A_{ij}^1\otimes\dots\otimes A_{ij}^m)_{i,j=1}^s
	\in\calL\Bigl(\bigodot_{p=1}^mC^p\Bigr).
\]
Note that the function $f_{ij}^{A^p}\in\calH_{A^p}$
vanishes for every point $i'\in E$ with $i'<i$.
For, if $k_{i'}^{A^p}$ is the reproducing kernel of $\calH_{A^p}$
for the point $i'\in E$, then, since $i'j'<ij$,
$f_{ij}^{A^p}$ is orthogonal to $k_{i'}^{A^p}(c_{i'j'}^p)$.
By the reproducing property, we have
\[
	\inner{f_{ij}^{A^p}(i'),c_{i'j'}^p}_{C_{i'}^p}=
	\inner{f_{ij}^{A^p},k_{i'}^{A^p}(c_{i'j'}^p)}_{\calH_{A^p}}=0,
\]
hence $f_{ij}^{A^p}(i')=0$, because the vectors $\{c_{i'j'}^p\}_{j'}$
span $C_{i'}^p$.

\begin{thm} \label{thm:main_ineq}
	In the above settings, suppose that $m\ge2$, $A^p\in\calL(C^p)$ is positive definite
	for each $p=1,\dots,m$ and that the set $\{k_i^{A^p}(c_{ij}^p)\}_{j=1}^{n_i^p}$
	is an orthogonal system for each $i$.
	Then, for each $\gam=(ij_1,\dots,ij_m)\in J_i$, the minimum norm
	$\lam_\gam^{\bigodot_{p=1}^mA^p}$ of IPIP	$(\bigodot_{p=1}^m\calH_{A^p},c,\gam)$
	satisfies the following inequality:
	\begin{align}\label{neq:lamOppenSchur}
		\lam_\gam^{\bigodot_{p=1}^mA^p}
		&\le\prod_{p=1}^m\lam_{ij_p}^{A^p}\cdot\Bigl\{\prod_{p=1}^m(\lam_{ij_p}^{A^p})^2
			\inner{A_{ii}^pc_{ij_p}^p,c_{ij_p}^p}_{\calH_{A^p}} \\
		&\qquad\qquad\qquad -\prod_{p=1}^m\bigl[(\lam_{ij_p}^{A^p})^2
			\inner{A_{ii}^pc_{ij_p}^p,c_{ij_p}^p}_{\calH_{A^p}}-1\bigr]\Bigr\}^{-1/2}.
	\end{align}
	Equality holds if and only if there exists a $l\le i$ such that $f_{ij_p}^{A^p}$
	is a linear combination of $k_i^{A^p}(c_{ij_p}^p)$
	and $\{k_l^{A^p}(c_{lj'}^p)\}_{lj'\le ij_p}$ for each $p=1,\dots,m$.
	In the case of equality, the minimum norm solution is given by
	\begin{equation}\label{eq:min sol}
		f_\gam^{\bigodot_{p=1}^mA^p}
		=\frac{\bigodot_{p=1}^m(\lam_{ij_p}^{A^p})^2k_i^{A^p}(c_{ij_p}^p)
			-\bigodot_{p=1}^m\{(\lam_{ij_p}^{A^p})^2k_i^{A^p}(c_{ij_p}^p)-f_{ij_p}^{A^p}\}}
			{\prod_{p=1}^m(\lam_{ij_p}^{A^p})^2
			\inner{A_{ii}^pc_{ij_p}^p,c_{ij_p}^p}_{\calH_{A^p}}-\prod_{p=1}^m\bigl[(\lam_{ij_p}^{A^p})^2
			\inner{A_{ii}^pc_{ij_p}^p,c_{ij_p}^p}_{\calH_{A^p}}-1\bigr]}.
	\end{equation}
\end{thm}

\begin{proof}
	By the reproducing property of $\bigodot_{p=1}^mk_i^{A^p}$,
	for each $f_p\in C^p$, $p=1,\dots,m$, we have
	\begin{align*}
		\inner{\bigodot_{p=1}^mf_p,
		\bigodot_{p=1}^mk_i^{A^p}(c_{ij_p}^p)}_{\bigodot_{p=1}^m\calH_{A^p}}
		=\prod_{p=1}^m\inner{f_p,k_i^{A^p}(c_{ij_p}^p)}_{\calH_{A^p}}.
	\end{align*}
	Let $f_{ij}^{A^p}$ be the minimum norm solution of IPIP
	$(\calH_{A^p},\{c_{ij}^p\}_{ij\in J_p},ij)$ and let $\lam_{ij}^{A^p}$
	be its norm.
	Consider the element $h$ of $\bigodot_{p=1}^m\calH_{A^p}$ defined by
	\begin{equation}\label{eq:h prod}
		h=\bigodot_{p=1}^m(\lam_{ij_p}^{A^p})^2k_i^{A^p}(c_{ij_p}^p)
		-\bigodot_{p=1}^m\{(\lam_{ij_p}^{A^p})^2k_i^{A^p}(c_{ij_p}^p)-f_{ij_p}^{A^p}\}.
	\end{equation}
	We show that $h$ satisfies the conditions for interpolation.
	First, we show that $h$ is orthogonal to $\bigodot_{p=1}^mk_{i'}^{A^p}(c_{i'j_p'}^p)$
	in $\bigodot_{p=1}^m\calH_{A^p}$ for every order $\gam'<\gam$ with
	$\gam'=(i'j_1',\dots,i'j_m')\in J_{i'}$.
	Since $\bigotimes_{p=1}^mk_{i'}^{A^p}(c_{i'j_p'}^p)$ is extremal,
	if $H$ is defined by
	\[
		H=
		\bigotimes_{p=1}^m(\lam_{ij_p}^{A^p})^2k_i^{A^p}(c_{ij_p}^p)
		-\bigotimes_{p=1}^m\{(\lam_{ij_p}^{A^p})^2k_i^{A^p}(c_{ij_p}^p)-f_{ij_p}^{A^p}\},
	\]
	then,
	\begin{equation}\label{eq:h=H}
		\inner{h,\bigodot_{p=1}^mk_{i'}^{A^p}(c_{i'j_p'}^p)}_{\bigodot_{p=1}^m\calH_{A^p}}
		=\inner{H,\bigotimes_{p=1}^mk_{i'}^{A^p}(c_{i'j_p'}^p)}_{\bigotimes_{p=1}^m\calH_{A^p}}.
	\end{equation}
	Since $i'\le i$, we divide the cases.

	Case $i'<i$:
	By expanding the second term of $H$ to the sum of simple tensors,
	each term of $H$ has at least one factor of the form $f_{ij_q}^{A^q}$,
	$1\le q\le m$.
	However, since $i'j_p'<ij_q$ for every $q$, by definition of $f_{ij_q}^{A^q}$
	we have $\inner{f_{ij_q}^{A^q},k_{i'}^{A^p}(c_{i'j_p'}^p)}=0$,
	So $H$ is orthogonal to $\bigotimes_{p=1}^mk_{i'}^{A^p}(c_{i'j_p'}^p)$.

	Case $i'=i$:
	Since $\gam'<\gam$, there exists a $p$ with $j_p'<j_p$ by definition
	of the lexicographic order.
	By expanding the second term of $H$, each term of simple tensors contains
	a factor $k_i^{A^p}(c_{ij_p}^p)$ or $f_{ij_p}^{A^p}$.
	By the hypothesis of the orthogonal system,
	$k_i^{A^p}(c_{ij_p}^p)\perp k_i^{A^p}(c_{ij_p'}^p)$.
	On the other hand, by definition $f_{ij_p}^{A^p}\perp f_{ij_p'}^{A^p}$.
	Thus, $H$ is orthogonal to $\bigotimes_{p=1}^mk_{i'}^{A^p}(c_{i'j_p'}^p)$, as desired.

	Second, we calculate the inner product of $h$ with $\bigodot_{p=1}^mk_i^{A^p}(c_{ij_p}^p)$
	by~\eqref{eq:h=H}:
	\begin{align}
		\inner{h,\bigodot_{p=1}^mk_i^{A^p}(c_{ij_p}^p)}_{\bigodot_{p=1}^m\calH_{A^p}}
		&=\prod_{p=1}^m(\lam_{ij_p}^{A^p})^2\inner{A_{ii}^pc_{ij_p}^p,c_{ij_p}^p}_{C_i^p} \\
		&\qquad\qquad	-\prod_{p=1}^m\bigl[(\lam_{ij_p}^{A^p})^2
				\inner{A_{ii}^pc_{ij_p}^p,c_{ij_p}^p}_{C_i^p}-1\bigr],
	\end{align}
	which is easily seen to be greater than or equal to 1, since
	$(\lam_{ij_p}^{A^p})^2\inner{A_{ii}^pc_{ij_p}^p,c_{ij_p}^p}_{C_i^p}\ge1$
	by Lemma~\ref{lem:main_ineq}.
	Thus,
	\begin{equation}\label{eq:ext fn}
		\Bigl\{\prod_{p=1}^m(\lam_{ij_p}^{A^p})^2\inner{A_{ii}^pc_{ij_p}^p,c_{ij_p}^p}
		-\prod_{p=1}^m\bigl[(\lam_{ij_p}^{A^p})^2\inner{A_{ii}^pc_{ij_p}^p,c_{ij_p}^p}
		-1\bigr]\Bigr\}^{-1}h
		\in P_\gam^{\bigodot_{p=1}^mA^p}.
	\end{equation}
	Also, by definition of $f_{ij_p}^{A^p}$ and $\lam_{ij_p}^{A^p}$,
	\begin{align}
		\|h\|^2_{\bigodot_{p=1}^m\calH_{A^p}}
		&\le\Bigl\|\bigotimes_{p=1}^m(\lam_{ij_p}^{A^p})^2k_i^{A^p}(c_{ij_p}^p)
			-\bigotimes_{p=1}^m\{(\lam_{ij_p}^{A^p})^2k_i^{A^p}(c_{ij_p}^p)-f_{ij_p}^{A^p}\}
			\Bigr\|^2_{\bigotimes_{p=1}^m\calH_{A^p}} \\
		&=\prod_{p=1}^m(\lam_{ij_p}^{A^p})^4\inner{A_{ii}^pc_{ij_p}^p,c_{ij_p}^p} \\
		&\qquad\qquad
			+\prod_{p=1}^m\Bigl\{(\lam_{ij_p}^{A^p})^4\inner{A_{ii}^pc_{ij_p}^p,c_{ij_p}^p}
			+(\lam_{ij_p}^{A^p})^2-2(\lam_{ij_p}^{A^p})^2\Bigr\} \\
		&\qquad\qquad -2\re\prod_{p=1}^m\Bigl\{(\lam_{ij_p}^{A^p})^2
			\biginner{k_i^{A^p}(c_{ij_p}^p),
				(\lam_{ij_p}^{A^p})^2k_i^{A^p}(c_{ij_p}^p)-f_{ij_p}^{A^p}}
			\Bigr\} \\
		&=\prod_{p=1}^m(\lam_{ij_p}^{A^p})^2
			\Bigl\{\prod_{p=1}^m(\lam_{ij_p}^{A^p})^2
			\inner{A_{ii}^pc_{ij_p}^p,c_{ij_p}^p}-\prod_{p=1}^m\bigl[(\lam_{ij_p}^{A^p})^2
			\inner{A_{ii}^pc_{ij_p}^p,c_{ij_p}^p}-1\bigr]\Bigr\}.
	\end{align}
	Therefore,
	\begin{align}
		\lam_\gam^{\bigodot_{p=1}^mA^p}
		&\le\|h\|_{\bigodot_{p=1}^m\calH_{A^p}} \\
		&\qquad\qquad \times
			\Bigl\{\prod_{p=1}^m(\lam_{ij_p}^{A^p})^2\inner{A_{ii}^pc_{ij_p}^p,c_{ij_p}^p}
			-\prod_{p=1}^m\bigl[(\lam_{ij_p}^{A^p})^2\inner{A_{ii}^pc_{ij_p}^p,c_{ij_p}^p}
			-1\bigr]\Bigr\}^{-1} \\
		&\le\frac{\prod_{p=1}^m\lam_{ij_p}^{A^p}}{\Bigl\{\prod_{p=1}^m(\lam_{ij_p}^{A^p})^2
			\inner{A_{ii}^pc_{ij_p}^p,c_{ij_p}^p}-\prod_{p=1}^m\bigl[(\lam_{ij_p}^{A^p})^2
			\inner{A_{ii}^pc_{ij_p}^p,c_{ij_p}^p}-1\bigr]\Bigr\}^{1/2}}.
	\end{align}
	This completes the proof of the inequality.

	Next, we consider the equality condition of the inequality.
	Noting the identity~\eqref{eq:h prod} and Proposition~\ref{prp:extremal},
	we see from the above proof of the inequality that the equality holds
	if and only if the simple tensor
	$\bigotimes_{p=1}^m\{(\lam_{ij_p}^{A^p})^2k_i^{A^p}(c_{ij_p}^p)-f_{ij_p}^{A^p}\}$
	is extremal.
	From Proposition~\ref{prp:tensCiDi}, this is equivalent to
	the condition that there exists a $l\in\{1,\dots,s\}$ such that,
	for each $p=1,\dots,m$, there exists a $c_l^p\in C_l^p$ satisfying
	\begin{align}
		(\lam_{ij_p}^{A^p})^2k_i^{A^p}(c_{ij_p}^p)-f_{ij_p}^{A^p}=k_l^{A^p}(c_l^p).
	\end{align}
	By Theorem~\ref{thm:inner_intp}, the minimum norm solution $f_{ij_p}^{A^p}$
	belongs to the span of the set $\{k_{i'}^{A^p}(c_{i'j'}^p)\colon i'j'\le ij_p\}$,
	and since $\calH_{A^p}$ is a direct sum of $\ran k_1^{A^p},\dots,\ran k_s^{A^p}$,
	we conclude that $l\le i$, and that, if $l=i$,
	$c_l^p\in\Span\{c_{ij'}^p\}_{j'\le j_p}$.
	If $f_{ij_p}^{A^p}$ is of the form
	\[
		f_{ij_p}^{A^p}
		=\gam k_i^{A^p}(c_{ij_p}^p)+\sum_{j'<j_p}\gam_{j'}k_i^{A^p}(c_{ij'}^p),
	\]
	then, by the interpolation condition, we obtain
	\[
		\gam
		=\inner{f_{ij_p}^{A^p},\gam k_i^{A^p}(c_{ij_p}^p)+\sum_{j'<j_p}\gam_{j'}k_i^{A^p}(c_{ij'}^p)}
		=\|f_{ij_p}^{A^p}\|^2=(\lam_{ij_p}^{A^p})^2,
	\]
	hence equality holds if and only if there exists a $l\le i$ such that,
	for every $p=1,\dots,m$, $f_{ij_p}^{A^p}$ is a linear combination of
	$k_i^{A^p}(c_{ij_p}^p)$ and $\{k_l^{A^p}(c_{lj'}^p)\}_{lj'\le ij_p}$.
	Moreover, if equality occurs, we see from the above proof that the
	function~\eqref{eq:ext fn} has the minimum norm for the IPIP.\@
	Therefore, this is the extremal function $f_\gam^{\bigodot_{p=1}^mA^p}$.
\end{proof}

\begin{rem}
	If $i=1$, the equality condition of the above Theorem is satisfied.
	Therefore, equality holds.
\end{rem}

\section{Application to determinant inequalities}

We derive determinant inequalities from the minimum norms $\lam_n$.
To evaluate the inequalities, we need to change the order of the products,
in which case the following elementary inequality is useful.
The case of $m=2$ has appeared in~\cite{Lin14}*{Prop.~2.1}.

\begin{lem}\label{lem:neq_elem}
	Let $m$ and $n$ be natural numbers.
	If every $a_{ij}\ge1$, then	the following inequality holds:
	\begin{equation}\label{eq:elem ineq}
		\prod_{i=1}^n\Bigl\{\prod_{j=1}^ma_{ij}-\prod_{j=1}^m(a_{ij}-1)\Bigr\}
		\ge\prod_{j=1}^m\prod_{i=1}^na_{ij}-\prod_{j=1}^m\Bigl(\prod_{i=1}^na_{ij}-1\Bigr).
	\end{equation}
	Equality occurs if and only if one of the conditions (i)--(iii) holds:
	\begin{enumerate}
		\item $m=1$ or $n=1$,
		\item There exists a column of the matrix $(a_{ij})$ such that every entry of
				the column is one, i.e.,~there exists a $j$ such that $a_{ij}=1$ for every $i$,
		\item There exists a row of the matrix $(a_{ij})$ such that
				every entry of the other rows is one, i.e.,~there exists $i_0$ such that
				$a_{ij}=1$ for every $i\ne i_0$ and every $j$.
	\end{enumerate}
\end{lem}

\begin{proof}
	Let $p_{ij}=a_{ij}-1$.
	Let $p_{i\blt}=(p_{i1},\dots,p_{im})\in\R_+^m$ and
	$p_{\blt j}=(p_{1j},\dots,p_{nj})\in\R_+^n$, respectively,
	be the $i$-th row and the $j$-th column of the matrix
	$p=(p_{ij})\in M_{n,m}(\R_+)$.
	Using multi-index notation, for the matrix $\al=(\al_{ij})\in M_{n,m}(\{0,1\})$,
	we denote by $p^\al$ the product $\prod_{i=1}^n\prod_{j=1}^mp_{ij}^{\al_{ij}}$.
	The matrix with each entry equal to one is denoted by $\1$.
	Let the order of the matrices be the product order with respect to each entry,
	that is, $(b_{ij})\ge(c_{ij})\iff b_{ij}\ge c_{ij}$ for each $i$ and $j$.
	Then,
	\begin{align*}
		\prod_{j=1}^ma_{ij}
		&=\prod_{j=1}^m(1+p_{ij})
			=\sum_{\al_{i\blt}}p_{i1}^{\al_{i1}}\cdots p_{im}^{\al_{im}}
			=\sum_{\al_{i\blt}}p_{i\blt}^{\al_{i\blt}}, \\
		\prod_{j=1}^m(a_{ij}-1)
		&=p_{i1}\dots p_{im}=p_{i\blt}^{\1},
	\end{align*}
	and we have identities:
	\begin{align*}
		\prod_{j=1}^ma_{ij}-\prod_{j=1}^m(a_{ij}-1)
		&=\sum_{\al_{i\blt}<\1}p_{i\blt}^{\al_{i\blt}}, \\
		\prod_{i=1}^n\Bigl\{\prod_{j=1}^ma_{ij}-\prod_{j=1}^m(a_{ij}-1)\Bigr\}
		&=\sum_{\fa i,\ \al_{i\blt}<\1}p_{1\blt}^{\al_{1\blt}}\dots p_{n\blt}^{\al_{n\blt}}
			=\sum_{\fa i,\ \al_{i\blt}<\1}p^\al, \\
		\prod_{j=1}^m\prod_{i=1}^na_{ij}
		&=\prod_{i=1}^n\prod_{j=1}^ma_{ij}
			=\sum_\al p_{1\blt}^{\al_{1\blt}}\dots p_{n\blt}^{\al_{n\blt}}
			=\sum_\al p^\al, \\
		\prod_{j=1}^m\Bigl(\prod_{i=1}^na_{ij}-1\Bigr)
		&=\prod_{j=1}^m\sum_{\al_{\blt j}>0}p_{\blt j}^{\al_{\blt j}}
			=\sum_{\fa j,\ \al_{\blt j}>0}p_{\blt1}^{\al_{\blt1}}\dots p_{\blt m}^{\al_{\blt m}}
			=\sum_{\fa j,\ \al_{\blt j}>0}p^\al.
	\end{align*}
	Thus, if we define the subsets $E$ and $F$ of the set $M_{n,m}(\{0,1\})$ by
	\begin{align}
		E
		&=\{\al\in M_{n,m}(\{0,1\})\colon\ex i\text{ s.t. }\al_{i\blt}=\1\}, \\
		F
		&=\{\al\in M_{n,m}(\{0,1\})\colon\fa j,\ \al_{\blt j}>0\},
	\end{align}
	then the inequality to prove is the following:
	\begin{equation}\label{neq:sum prod}
		\sum_{\al\in E}p^\al\le\sum_{\al\in F}p^\al.
	\end{equation}
	However, it is clear that $E\subset F$, and since $p^\al\ge0$,
	we have proved the inequality~\eqref{eq:elem ineq}.

	Next, we consider the case of equality.
	The condition for equality is $p^\al=0$ for every $\al\in F\sm E$,
	where
	\[
		F\sm E
		=\{\al\in M_{n,m}(\{0,1\})\colon\ \al_{i\blt}<\1\text{ and }\al_{\blt j}>0
		\text{ for every $i$ and $j$}\}.
	\]
	Let us divide the cases as follows:
	\begin{enumerate}[(a)]
		\item $p_{\bullet j}=0$ for some $j$:
				In this case, we have $p^\al=0$ for every $\al\in F$.
				Thus, $\sum_{\al\in E}p^\al=\sum_{\al\in F}p^\al=0$,
				hence equality holds.
		\item $p_{\blt j}\ne0$ for every $j$: We further divide this division into cases.
			\begin{enumerate}[1.]
				\item $m=1$ or $n=1$: Then, the inequality~\eqref{eq:elem ineq} is an equality.
				\item The case where there exists an $i_0$ such that $p_{ij}=0$ for every
						$i\ne i_0$ and every $j$:
						Since, every $i$-th row, $i\ne i_0$, of the matrix $p$ is 0,
						the inequality~\eqref{eq:elem ineq} reduces to the case $n=1$,
						in which case equality holds.
				\item Otherwise:
						In this case, for every $j=1,\dots,m$, there exists an $i_j$ such that
						$p_{i_jj}\ne0$.
						Since $m\ge2$, we can choose the set $\{i_j\}_{j=1}^m$ with cardinality
						greater than or equal to two.
						Then, the matrix $\al=(\al_{ij})$ defined by
						\[
							\al_{ij}=\begin{cases}
								1, & (i=i_j) \\
								0, & (i\ne i_j)
							\end{cases}
						\]
						satisfies $\al\in F\sm E$, hence $p^\al\ne0$.
						Therefore, equality does not hold in this case.
			\end{enumerate}
	\end{enumerate}
	From the above considerations, we conclude that equality holds if and only if
	in case (a), (b)-1, or (b)-2.
	This completes the proof of the equality conditions.
\end{proof}

As in the list at the beginning of \S\ref{sec:main results},
to obtain Oppenheim-Schur's inequality for block matrices $A^p$,
we assume the following:
\begin{itemize}
	\item $C^p=\bigoplus_{i=1}^sC_i^p$,\quad$C_i^p=\C^{n_{ip}}$,
			\quad$n_p=\dim C^p=\sum_{i=1}^sn_{ip}$,\quad $p=1,\dots,m$,
	\item $A^p=(A_{ij}^p)_{i,j=1}^s$,\quad$A_{ij}^p\in M_{n_{ip},n_{jp}}(\C)$,
			where $A^p$ is a $s\times s$ symmetric partitioned positive semidefinite
			block matrix ($p=1,\dots,m$).
			By a {\em symmetric partitioned} block matrix, we mean that the main-diagonal
			blocks are square matrices.
	\item $\{c_{ij}^p\}_{j=1}^{n_{ip}}$:
			a CONS of $\C^{n_{ip}}$ consisting of eigenvectors of $A_{ii}^p\ge0$	($i=1,\dots,s$).
\end{itemize}

For each $i=1,\dots,s$, $\{\bigotimes_{p=1}^mc_{ij_p}^p\colon j_p=1,\dots,n_{ip}\}$
is a CONS consisting of eigenvectors of the matrix $\bigotimes_{p=1}^mA_{ii}^p$
(\cite{HornJohnson91}*{p.~245}).
Let $\calH_{A^p}$ ($p=1,\dots,m$) be the $C^p$-RKHS on $E$ with
the reproducing kernel $A^p\in\calL(C^p)$.
Then, the reproducing kernel of the Hadamard product RKHS $\bigodot_{p=1}^m\calH_{A^m}$
is given by the {\em Hadamard product} of $A^1$, \dots, $A^m$:
\[
	\bigodot_{p=1}^mA^p=A^1*\dots*A^m
	=(A_{ij}^1\otimes\dots\otimes A_{ij}^m)_{i,j=1}^s,
\]
which is also called the {\em Khatri-Rao product} of the block matrices
$A^1,\dots,A^m$.
We denote by $(X)_i$ the $i$-th leading principal block submatrix of a block matrix $X$,
where we define $(X)_0=1$ for simplicity's sake.

\begin{thm}\label{thm:sub ineq}
	In the above settings, if we assume that each block matrix $A^1,\dots$, $A^m$
	is positive definite,	then their Hadamard product $\bigodot_{p=1}^mA^p$ satisfies
	the following inequality for each $i=1,\dots,s$:
	\begin{align}\label{eq:main th1}
		\frac{|(\bigodot_{p=1}^mA^p)_i|}{|(\bigodot_{p=1}^mA^p)_{i-1}|}
		&\ge\prod_{p=1}^m|A^p_{ii}|^{\sig_{ip}}
			-\prod_{p=1}^m\Bigl\{|A^p_{ii}|^{\sig_{ip}}
				-\Bigl(\frac{|(A^p)_i|}{|(A^p)_{i-1}|}\Bigr)^{\sig_{ip}}\Bigr\},
	\end{align}
	with $\sig_{ip}=n_{ip}^{-1}\prod_{q=1}^mn_{iq}$.
	Equality occurs if and only if one of the conditions (a)--(c) holds:
	\begin{enumerate}[(a)]
		\item $m=1$ or $i=1$.
		\item There exists a $p$ with $1\le p\le m$ such that $A_{li}^p$ and $A_{il}^p$ is 0
				for every $l<i$.
		\item $n_{ip}=1$ for every $p$, and there exists an $i_0$ with $i_0<i$ such that
				$A^p_{li}=A^p_{li_0}(A^p_{i_0i_0})^{-1}A^p_{i_0i}$ for every $l$ with $l<i$
				and every $p$.
	\end{enumerate}
\end{thm}

\begin{proof}
	If $m=1$ or $i=1$, then it is clear that equality holds.
	Thus, we assume that $m\ge2$ and $i\ge2$.
	Since the set $\{c_{ij}^p\}_{j=1}^{n_{ip}}$ is a CONS of $C_i^p$ consisting
	of eigenvectors of the matrix $A_{ii}^p\in\calL(C_i^p)$,
	we have	$|A_{ii}^p|=\prod_{j=1}^{n_{ip}}\xi(c_{ij}^p)$,
	where $\xi(v)$ denotes the eigenvalue of an eigenvector $v$.
	If $\gam=(ij_1,\dots,ij_m)\in J_i$, then we have $\al_{\gam p}\ge1$
	by Lemma~\ref{lem:main_ineq},
	where $\al_{\gam p}$ is defined by
	\begin{equation}\label{eq:lam_xi}
		\al_{\gam p}=(\lam_{ij_p}^{A^p})^2\xi(c_{ij_p}^p).
	\end{equation}
	Applying Lemma~\ref{lem:prod det}, Theorem~\ref{thm:main_ineq}
	and Lemma~\ref{lem:neq_elem}, we have
	\begin{align}
		\frac{|(\bigodot_{p=1}^mA^p)_i|}{|(\bigodot_{p=1}^mA^p)_{i-1}|}
		&=\prod_{\gam\in J_i}(\lam_\gam^A)^{-2} \\
		&\ge\prod_{\gam\in J_i}\frac{\prod_{p=1}^m(\lam_{ij_p}^{A^p})^2 \label{eq:sub ineq1}
			\xi(c_{ij_p}^p)-\prod_{p=1}^m\bigl[(\lam_{ij_p}^{A^p})^2
			\xi(c_{ij_p}^p)-1\bigr]}
			{\prod_{p=1}^m(\lam_{ij_p}^{A^p})^2} \\
		&=\frac{\prod_{\gam\in J_i}\bigl\{\prod_{p=1}^m(\lam_{ij_p}^{A^p})^2
			\xi(c_{ij_p}^p)-\prod_{p=1}^m\bigl[(\lam_{ij_p}^{A^p})^2
			\xi(c_{ij_p}^p)-1\bigr]\bigr\}}
			{\prod_{\gam\in J_i}\prod_{p=1}^m(\lam_{ij_p}^{A^p})^2} \\
		&\ge\frac{\prod_{p=1}^m\prod_{\gam\in J_i}(\lam_{ij_p}^{A^p})^2\xi(c_{ij_p}^p)
			-\prod_{p=1}^m\{\prod_{\gam\in J_i}(\lam_{ij_p}^{A^p})^2\xi(c_{ij_p}^p)-1\}}
			{\prod_{p=1}^m\prod_{\gam\in J_i}(\lam_{ij_p}^{A^p})^2}. \label{eq:sub ineq}
	\end{align}
	If $\pi_p\colon J=\prod_{j=1}^mJ^j\to J^p$ denotes the projection, then
	the cardinality of the set $\pi_p^{-1}(ij_p)$ is $\sig_{ip}$.
	Thus,
	\begin{align}
		\prod_{\gam\in J_i}\xi(c_{ij_p}^p)
		&=\prod_{\pi_p^{-1}(ij_p)}\prod_{ij_p\in J^p}\xi(c_{ij_p}^p)
			=|A^p_{ii}|^{\sig_{ip}}.
		\intertext{Similarly,}
		\prod_{\gam\in J_i}(\lam_{ij_p}^{A^p})^2 \label{eq:lam^2}
		&=\Bigl(\frac{|(A^p)_{i-1}|}{|(A^p)_i|}\Bigr)^{\sig_{ip}}.
	\end{align}
	Therefore,
	\begin{equation}
		\prod_{\gam\in J_i}(\lam_{ij_p}^{A^p})^2\xi(c_{ij_p}^p) \label{eq:lam^2xi}
		=\Bigl(\frac{|A^p_{ii}||(A^p)_{i-1}|}{|(A^p)_i|}\Bigr)^{\sig_{ip}}.
	\end{equation}
	Substituting these identities into~\eqref{eq:sub ineq},
	we obtain the inequality~\eqref{eq:main th1}, as desired.

	Now, we consider the equality condition of the inequality.
	From the above proof, we see that equality occurs when the following conditions
	(I) and (II) are satisfied:
	\begin{enumerate}[(I)]
		\item The equality conditions (ii) and (iii) of	Lemma~\ref{lem:neq_elem}
				for the matrix $(\al_{\gam p})$ of~\eqref{eq:lam_xi}.
		\item The equality condition of Theorem~\ref{thm:main_ineq} for each $\gam\in J_i$.
	\end{enumerate}

	First, we consider case (I) of the equality condition (ii).
	Then, there exists a $p$ such that $(\lam_{ij}^{A^p})^2\xi(c_{ij}^p)=1$ for
	every $j$.
	By Lemma~\ref{lem:main_ineq}, for every $l<i$, the $C_l^p$-component of $A^pc$
	vanishes for every $c\in C_i^p$.
	Thus, $A_{li}^p=A_{il}^p=0$ for every $l<i$, since $A^p$ is Hermite.
	This is the equality condition (b) of Theorem.
	Conversely, it is easy to see that equality holds if the condition (b) is satisfied.

	Next, we consider case (I) of the equality condition (iii).
	By Lemma~\ref{lem:neq_elem} there exists a $\gam_0=(ij_1',\dots,ij_m')\in J_i$
	such that $\al_{\gam p}=1$ for every $\gam=(ij_1,\dots,ij_m)\in J_i\sm\{\gam_0\}$
	and for every $p=1,\dots,m$.
	We divide the cases according to the integers $n_{ip}\ (=\dim C_i^p)$, $p=1,\dots,m$.
	If there exists a $q$ with $n_{iq}>1$, then we can choose an index $j_q$
	with $j_q\ne j_q'$ such that the index $\gam=(ij_1,\dots,ij_q,\dots,ij_m)\in J_i$
	satisfies $\gam\ne\gam_0$ for every index $ij_p$ ($p\ne q$).
	Thus, $\al_{\gam p}=1$.
	By Lemma~\ref{lem:main_ineq} we have $A_{li}^p=A_{il}^p=0$ for every $l<i$
	and every $p\ne q$.
	Therefore, in this case the condition (b) holds.
	The remaining case is one where $n_{ip}=1$ for every $p$,
	in which case the inequality~\eqref{eq:sub ineq} is always an equality,
	since only one order is possible.
	As for the inequality~\eqref{eq:sub ineq1}, by the equality condition of
	Theorem~\ref{thm:main_ineq}, there exists an $i_0\ (\le i)$ such that
	$f_{ij}^{A^p}$ is a linear combination of $k_i^{A^p}(u_{ij}^p)$ and
	$k_{i_0}^{A^p}(c_p)$, $c_p\in C_{i_0}^p$ for each $p$.
	If $i_0=i$, then $f_{ij}^{A^p}$ is a constant multiple of $k_i^{A^p}(u_{ij}^p)$,
	since $C_i^p$ is one dimensional.
	Similar to the proof of Lemma~\ref{lem:main_ineq},
	we have $A_{li}^p=A_{il}^p=0$ for each $l<i$, which is the case (b).
	If $i_0<i$, then $f_{ij}^{A^p}$ is a nonzero multiple of
	$k_i^{A^p}(u_{ij}^p)+k_{i_0}^{A^p}(c_p)$.
	By the orthogonality condition of $f_{ij}^{A^p}$, we have
	\[
		0=\inner{k_i^{A^p}(u_{ij}^p)+k_{i_0}^{A^p}(c_p),k_l^{A^p}(u_{lj'})}_{\calH_{A^p}}
		=\inner{A^p_{li}u_{ij}^p+A^p_{li_0}c_p,u_{lj'}}_{C_l^p}
	\]
	for each $l<i$ and for each $j'$.
	Since the set $\{u_{lj'}\}_{j'}$ spans $C_l^p$, $A^p_{li}u_{ij}^p+A^p_{li_0}c_p=0$
	for each $l<i$.
	Noting that $u_{ij}^p\in C_i^p$ and $C_i^p$ is one dimensional,
	we have	$A^p_{li}=A^p_{li_0}(A_{i_0i_0}^p)^{-1}A^p_{i_0i}$ for each $l<i$.
	This is the condition (c) of Theorem.
	Conversely, it is easy to see from the above proof that
	if (c) is satisfied, then equality holds.
\end{proof}

\begin{rem}
	Since $\al_{\gam p}\ge1$, from~\eqref{eq:lam^2xi} we have, for $i=1,\dots,s$,
	\[
		|A_{ii}|\ge\frac{|(A)_i|}{|(A)_{i-1}|},
	\]
	where $A=(A_{ij})$ is a symmetric partitioned positive definite
	$s\times s$	block matrix.
	Multiplying this inequality for $i=1,\dots,s$ gives Fischer's inequality
	(cf.~\cite{HornJohnson85}*{p.~506}):
	\[
		\prod_{i=1}^s|A_{ii}|\ge|A|.
	\]
\end{rem}

In particular, if the matrix $A^p$ is a $s\times s$ block matrix of the form
$A^p\in M_s(M_{t_p})$ for each $p=1,\dots,m$,
we obtain the following extension of Oppenheim-Schur's inequality.
The cases $m=2$ and $t_1=t_2=1$ correspond to the usual Oppenheim-Schur's inequality.

\begin{thm}
	If each block matrix $A^p=(A_{ij}^p)\in M_s(M_{t_p})$, $p=1,\dots,m$,
	is positive semidefinite, then we have the following inequality for
	the determinant of the Hadamard product $\bigodot_{p=1}^mA^p$:
	\begin{align} \label{eq:main ineq}
		|\bigodot_{p=1}^mA^p|
		&\ge\prod_{p=1}^m\prod_{i=1}^s|A^p_{ii}|^{\sig_p}
			-\prod_{p=1}^m\Bigl\{
			\prod_{i=1}^s|A^p_{ii}|^{\sig_p}
			-|A^p|^{\sig_p}\Bigr\},
	\end{align}
	with $\sig_p=t_p^{-1}\prod_{q=1}^mt_q$.
	If each $A^p$ is positive definite, then equality holds if and only if
	one of the following conditions holds:
	\begin{enumerate}[(a)]
		\item $m=1$ or $s=1$.
		\item There exists $p$ such that $A^p$ is block diagonal.
		\item For every $p=1,\dots,m$, $t_p=1$ (i.e., $A^p\in M_s$) and
				there exist $i$ and $j$ ($1\le i,j\le s$) such that
				every entry of the matrix $A^p$ is 0 except for the diagonal entries
				and the $(i,j)$ and $(j,i)$ entries.
	\end{enumerate}
\end{thm}

\begin{proof}
	If $A^p$ is positive semidefinite, then the inequality~\eqref{eq:main ineq}
	is obtained by taking the limit $\eps\to+0$ of one for the positive definite
	matrices $A^p+\eps I$, where $I$ denotes the identity matrix.
	Thus, without loss of generality, we can assume that every $A^p$ is positive definite.
	Moreover, if $m=1$, then equality holds trivially in~\eqref{eq:main ineq},
	and if $s=1$, then~\eqref{eq:main ineq} is also equality,
	since this is the case $i=1$ of Theorem~\ref{thm:sub ineq}.
	Thus, we assume $m\ge2$ and $s\ge2$.
	Similar to the proof of Theorem~\ref{thm:sub ineq}, we have
	\begin{align}
		\frac{|(\bigodot_{p=1}^mA^p)_i|}{|(\bigodot_{p=1}^mA^p)_{i-1}|}
		&\ge\prod_{p=1}^m|A^p_{ii}|^{\sig_p}
			-\prod_{p=1}^m\Bigl\{|A^p_{ii}|^{\sig_p}
				-\Bigl(\frac{|(A^p)_i|}{|(A^p)_{i-1}|}\Bigr)^{\sig_p}
			\Bigr\} \\
		&=\prod_{p=1}^m\Bigl(\frac{|(A^p)_i|}{|(A^p)_{i-1}|}\Bigr)^{\sig_p} \\
		&\qquad\times
		\Bigl[
			\prod_{p=1}^m\Bigl(\frac{|A^p_{ii}||(A^p)_{i-1}|}{|(A^p)_i|}\Bigr)^{\sig_p}
			-\prod_{p=1}^m
				\Bigl\{
					\Bigl(\frac{|A^p_{ii}||(A^p)_{i-1}|}{|(A^p)_i|}\Bigr)^{\sig_p}-1
				\Bigr\}
		\Bigr].
	\end{align}
	Thus,
	\begin{align}
		|\bigodot_{p=1}^mA^p|
		&=\prod_{i=1}^s\frac{|(\bigodot_{p=1}^mA^p)_i|}{|(\bigodot_{p=1}^mA^p)_{i-1}|} \\
			\label{eq:tochu}
		&\ge\prod_{i=1}^s\prod_{p=1}^m\Bigl(\frac{|(A^p)_i|}{|(A^p)_{i-1}|}\Bigr)^{\sig_p} \\
		&\qquad\times\prod_{i=1}^s
			\Bigl[
				\prod_{p=1}^m\Bigl(\frac{|A^p_{ii}||(A^p)_{i-1}|}{|(A^p)_i|}\Bigr)^{\sig_p}
				-\prod_{p=1}^m
					\Bigl\{
						\Bigl(\frac{|A^p_{ii}||(A^p)_{i-1}|}{|(A^p)_i|}\Bigr)^{\sig_p}-1
					\Bigr\}
			\Bigr].
	\end{align}
	Consequently, from Lemma~\ref{lem:neq_elem}, we have
	\begin{align}
		|\bigodot_{p=1}^mA^p|
		&\ge\prod_{p=1}^m\prod_{i=1}^s\Bigl(\frac{|(A^p)_i|}{|(A^p)_{i-1}|}\Bigr)^{\sig_p} \\
		&\qquad\times
			\Bigl[
				\prod_{p=1}^m\prod_{i=1}^s\Bigl(\frac{|A^p_{ii}||(A^p)_{i-1}|}{|(A^p)_i|}\Bigr)^{\sig_p}
				-\prod_{p=1}^m
					\Bigl\{
						\prod_{i=1}^s\Bigl(\frac{|A^p_{ii}||(A^p)_{i-1}|}{|(A^p)_i|}\Bigr)^{\sig_p}-1
					\Bigr\}
					\Bigr] \\
		&=\prod_{p=1}^m|A^p|^{\sig_p}
			\times\Bigl[\prod_{p=1}^m\frac{\prod_{i=1}^s|A_{ii}^p|^{\sig_p}}{|A^p|^{\sig_p}}
				-\prod_{p=1}^m\Bigl\{
					\frac{\prod_{i=1}^s|A_{ii}^p|^{\sig_p}}{|A^p|^{\sig_p}}-1\Bigr\}
				\Bigr] \\
		&=\prod_{p=1}^m\prod_{i=1}^s|A^p_{ii}|^{\sig_p}
		-\prod_{p=1}^m\Bigl\{
		\prod_{i=1}^s|A^p_{ii}|^{\sig_p}
		-|A^p|^{\sig_p}\Bigr\},
	\end{align}
	as desired.

	Now, we consider the equality condition of the inequality~\eqref{eq:main ineq}.
	If we define $\al_{\gam p}$ by~\eqref{eq:lam_xi},
	then the equality condition is derived from those of (ii) and (iii) of
	Lemma~\ref{lem:neq_elem} and those of Theorem~\ref{thm:main_ineq}.
	Consider the case (ii).
	In this case, there exists a $p$ such that $(\lam_{ij}^{A^p})^2\xi(c_{ij}^p)=1$
	for every $j$.
	By Lemma~\ref{lem:main_ineq}, $k_{i'}^{A^p}(c_{i'j'})\perp k_i^{A^p}(c_{ij}^p)$
	for every $i'<i$ and every $j'$.
	Thus,
	\[
		\inner{A_{i'i}^pc_{ij}^p,c_{i'j'}^p}_{C_{i'}}
		=\inner{k_i^{A^p}(c_{ij}^p),k_{i'}^{A^p}(c_{i'j'}^p)}_{\calH_{A^p}}=0.
	\]
	Since $\{c_{i'j'}^p\}$ and $\{c_{ij}^p\}$ span $C_{i'}^p$ and $C_i^p$, respectively,
	we have $A_{i'i}^p=0$.
	Therefore, for some $p$, $A^p$ is an upper triangular block matrix,
	which implies that $A^p$ is a block diagonal matrix, since $A^p$ is Hermitian.
	Conversely, if this is the case, then $\bigodot_{p=1}^mA^p$ is also block diagonal,
	and thus it is easy to see that the inequality~\eqref{eq:main ineq} is an equality.

	In case (iii), there exists a $\gam_0=(ij_1,\dots,ij_m)\in J$ with $\al_{\gam p}=1$
	for every $\gam=(i'j_1',\dots,i'j_m')\in J\sm\{\gam_0\}$ and every $p=1,\dots,m$.
	We divide the cases.
	If there exists a $q$ with $t_q>1$, then, as in the proof of the equality condition
	of Theorem~\ref{thm:sub ineq}, $A^p$ is block diagonal for every $p\ne q$,
	which implies that the equality is trivial.
 
	There remains the case that $t_p=1$ for every $p$, that is, the case of scalar matrices
	or non-block matrices.
	In this case, by the equality condition of Lemma~\ref{lem:main_ineq},
	the $j$-th column of $A^p$ satisfies $A_{kj}^p=A_{jk}^p=0$ ($k=1,\dots,j-1$)
	for every $j\ (\ne i)$.
	From this and the equality condition of Theorem~\ref{thm:main_ineq},
	we easily conclude that there exists an $i_0<i$ such that
	$A_{li}^p=0$ for every $l\ne i_0$ and $l\ne i$.
	Therefore, $A^p$ is the matrix which satisfies the condition (b) of Theorem
	for every $p$.
	Conversely, if this is the case, then
	\[
		|A^p|=(A_{ii}^pA_{jj}^p-|A_{ij}^p|^2)\prod_{k\ne i,j}^sA_{kk}^p.
	\]
	Thus, equality holds in this case (cf.~\cite{Oppenheim30}).
\end{proof}

\begin{rem}
	If the matrices $A^p$ are all positive definite,
	then, by the inequality~\eqref{eq:tochu}, we have
	\begin{align}
		|\bigodot_{p=1}^mA^p|
		&\ge\prod_{p=1}^m|A^p|^{\sig_p} \\
		&\qquad\times\prod_{i=1}^s\Bigl\{
			\prod_{p=1}^m\Bigl(\frac{|A^p_{ii}||(A^p)_{i-1}|}{|(A^p)_i|}
				\Bigr)^{\sig_p}-\prod_{p=1}^m
				\Bigl[
				\Bigl(\frac{|A^p_{ii}||(A^p)_{i-1}|}{|(A^p)_i|}\Bigr)^{\sig_p}.
				-1\Bigr]\Bigr\}
	\end{align}
	Noting that $A_{11}^p=(A^p)_1$ and $(A^p)_0=1$,
	we see that the product $\prod_{i=1}^s$ can be changed to $\prod_{i=2}^s$.
	Thus, for $m\ge3$ the inequality~\eqref{eq:tochu} is sharper than
	the one obtained by~Li-Feng~\cite{LiFeng21}*{Theorem 2.5}.
\end{rem}

\input{OppenSchur.bbl}

\end{document}